\documentclass[preprint,12pt]{elsarticle}
\usepackage{amsmath,amssymb,amsthm}
\usepackage[
  colorlinks=true,
  linkcolor=blue,
  citecolor=blue,
  urlcolor=blue
]{hyperref}
\usepackage[T1]{fontenc}
\usepackage{lmodern}
\newtheorem{lemma}{Lemma}
\newtheorem{theorem}{Theorem}

\biboptions{numbers,sort&compress}

\begin{document}
\begin{frontmatter}
\title{Short Quartic Exponential Sums in an Intermediate Range and Waring's Problem for Fourth Powers with Almost Equal Summands}
\author[tsufe]{Karimjon Ibrohimjonovich Mirzoabdughafurov\corref{cor1}}
\ead{karimjon2003@mail.ru}
\cortext[cor1]{Corresponding author.}
\address[tsufe]{Tajik State University of Finance and Economics, Dushanbe, Tajikistan}
\begin{abstract}
This paper is devoted to short quartic Weyl sums in an intermediate
range of rational approximations and to their application to Waring's
problem with almost equal summands. We obtain a pointwise
estimate for such sums and use it, together with the Hardy--Littlewood
circle method, to establish an asymptotic formula for the number of
representations of a sufficiently large natural number as a sum of
seventeen fourth powers of integers lying in a short interval about a
common centre. The asymptotic formula is valid for
$N^{6/25+\varepsilon}\leq H\leq N^{1/4-\varepsilon}$ and improves the
previously known admissible lower bound for $H$.
\end{abstract}
\begin{keyword}
Waring's problem \sep almost equal summands \sep fourth powers
\sep short exponential sums \sep Hardy--Littlewood circle method.
\MSC[2020] 11P05 \sep 11L15 \sep 11P55
\end{keyword}
\end{frontmatter}

Notation: $e(z)=e^{2\pi i z}$, $T(\alpha;x,y)=\sum\limits_{x-y<n\leq x} e(\alpha n^4)$,
$$
\gamma(\lambda;x,y)=\int\limits_{-\frac{1}{2}}^{\frac{1}{2}}e\Bigl(\lambda(x-\frac{y}{2}+yt)^4\Bigr)dt,\quad 
S(a,q)=\sum\limits_{n=1}^{q}e\Big(\frac{an^4}{q}\Big),
$$
\[
\mathfrak{S}(N)
 =\sum_{q=1}^{\infty}
  \sum_{\substack{0\le a<q\\(a,q)=1}}
  \frac{S^{17}(a,q)}{q^{17}}
  e\!\left(-\frac{aN}{q}\right).
\]

\section{Introduction}

The classical Waring problem is one of the central problems of additive
number theory. It concerns representations of a natural number as a
sum of a fixed number of like powers of natural numbers. The
Hardy--Littlewood method, together with the subsequent development of
I.~M.~Vinogradov's method of exponential sums, made it possible to
obtain not only solvability results for the corresponding additive
equations but also asymptotic formulae for the number of their
solutions; see, for example,
\cite{vinogradov,KarOATCh,Von,Archipov}.

One natural refinement of the classical Waring problem is obtained by
imposing additional restrictions on the summands. In particular,
Waring's problem with almost equal summands requires all variables to
belong to a short interval about the same point. This formulation is
more delicate than the unrestricted Waring problem, since the length
of the admissible interval is substantially smaller than the natural
scale of variation of the variables.

For cubes with almost equal summands, an asymptotic formula for the
number of representations of a natural number as a sum of nine cubes
was obtained in \cite{MirzoabdughafurovWaring2008}. Localized versions
of Waring's problem for sums of almost equal powers were investigated
by Daemen, who obtained an asymptotic formula and related estimates
for Weyl sums \cite{DaemenAsymptotic,DaemenLower}. Related
Waring--Goldbach problems with almost equal prime summands were studied
by Wei and Wooley \cite{WeiWooley}.

For fourth powers with almost equal summands, an asymptotic formula
for the number of representations of a natural number as a sum of
seventeen fourth powers was obtained in \cite{Azamov2011}. That paper
considered the number of solutions of the equation
\begin{equation}
\label{eq:waring17}
        x_1^4+x_2^4+\cdots+x_{17}^4=N
\end{equation}
under the conditions
\begin{equation}
\label{eq:almost-equal}
        |x_i-N_1|\leq H,\qquad i=1,\ldots,17,\qquad
        N_1=\left(\frac{N}{17}\right)^{1/4}.
\end{equation}
and established an asymptotic formula when
\[
        H\geq N^{13/54+\varepsilon}.
\]

Related questions have also been considered in Waring's problem with
almost proportional summands, where restrictions are imposed on the
quantities $x_i^n$ relative to fixed proportions of the number $N$.
In this direction, asymptotic formulae were obtained, including for
the case of $r=2^n+1$ summands \cite{Rakhmonov2024}. For a special
choice of the proportions, such problems include Waring's problem
with almost equal summands as a particular case.

More precisely, in the case
\[
        n=4,\qquad r=17,\qquad
        \mu_1=\mu_2=\cdots=\mu_{17}=\frac{1}{17},
\]
the bound obtained in \cite{Rakhmonov2024} corresponds to the
condition
\[
        H\geq N^{13/54+\varepsilon}
\]
for seventeen almost equal fourth powers. Thus, in the special case
considered in the present paper, this result gives the previously
known admissible lower bound for $H$.

An essential ingredient in the study of localized Waring problems is
the estimation of short Weyl sums. Earlier estimates for short
quartic Weyl sums in neighbourhoods of rational points were obtained
in
\cite{RakhmonovAzamovMirzo2010}.
Writing
\[
        \alpha=\frac{a}{q}+\lambda,\qquad (a,q)=1,
\]
and assuming that
\[
        \tau\geq 24x^2y,
\]
those results provide an asymptotic approximation in the range
\[
        |\lambda|\leq \frac{1}{8qx^3}
\]
and the estimate
\[
        T(\alpha;x,y)
        \ll q^{3/4}\log q+q^{1/4}x^{1/2}
\]
in the adjacent range
\[
        \frac{1}{8qx^3}
        <|\lambda|
        \leq\frac{1}{q\tau}.
\]
In the circle-method argument below, these results are used to estimate
the contributions from the sets $E_{11}$ and $E_{12}$.

However, after a second application of Dirichlet's rational
approximation lemma, the estimation of the remaining set $E_2$ leads
to a different range of rational approximations. In particular, the
part $E_{22}$ corresponds to relatively small denominators and to
values of $\lambda$ satisfying
\[
        \frac{1}{qx^2y}\ll |\lambda|
        \ll\frac{1}{qy^3}.
\]
This intermediate range extends beyond the upper restriction
\[
        |\lambda|\leq\frac{1}{q\tau}
\]
in the estimates used for the major arcs and requires a different
pointwise treatment.

The first purpose of the present paper is to establish an estimate for
a short quartic Weyl sum in this intermediate range. The proof is
based on an additive Fourier expansion of the rational part of the
exponential sum and an analysis of the resulting oscillatory
integrals. The dual parameters are separated according to whether the
corresponding phase has a stationary point in the interval of
integration. This yields the following estimate.

\begin{theorem}\label{thm:quartic-sum}
Let
\[
x\ge x_0>0,\quad 0<y\le 0.01x,\quad
\alpha=\frac{a}{q}+\lambda,\quad  (a,q)=1,
\]
and suppose that
\[
\frac{1}{qx^2y}\ll |\lambda|\ll \frac{1}{qy^3}.
\]
Then
\[
T\left(\frac{a}{q}+\lambda;x,y\right)
\ll
q^{\frac{3}{4}}\ln q
+
q^{\frac{3}{4}}\ln(2+|\lambda|x^3)
+
q^{\frac{1}{4}}\frac{x}{\sqrt{y}}.
\]
\end{theorem}

The pointwise estimate in Theorem~\ref{thm:quartic-sum}
provides the bound required for the intermediate range arising in
the circle-method analysis of $J(N,H)$. As a consequence, the
previously known condition
\[
        H\geq N^{13/54+\varepsilon}
\]
can be replaced by
\[
        H\geq N^{6/25+\varepsilon}.
\]

This leads to our second main result.
\begin{theorem}\label{varing4}
Let $N$ be a sufficiently large natural number and let
$\varepsilon>0$ be a sufficiently small fixed number. Then, for
\[
        N^{6/25+\varepsilon}
        \leq H
        \leq N^{1/4-\varepsilon},
\]
the number $J(N,H)$ of integer solutions of
\eqref{eq:waring17} satisfying \eqref{eq:almost-equal} has the
asymptotic formula
\begin{align*}
J(N,H)&= \frac{B\mathfrak{S}(N)H^{16}}{N^{\frac{3}{4}}}
+O\left(\frac{H^{16}}{N^{\frac{3}{4}}L^{16}}\right),
\end{align*}
where \(L=\log N\), \(B=\frac{\sqrt[4]{17^3}}{4\cdot16!}\sum\limits_{k=0}^{8}(-1)^k\binom{17}{k}(17-2k)^{16}\) is a positive absolute
constant, and \(\mathfrak{S}(N)\) is the singular series,
which is bounded below by a positive constant.
\end{theorem}

The latter assertion follows directly from \cite[Theorem~4.6]{Von}.

The proof of Theorem~\ref{varing4} is based on the
Hardy--Littlewood circle method. The contributions from the major
arcs are treated by previously known estimates, while
Theorem~\ref{thm:quartic-sum} supplies the additional saving required
for the intermediate range occurring in the minor-arc analysis.

\section{Auxiliary lemmas}
\begin{lemma}\label{int1} {\rm \cite{KarOATCh}.} For every real $\kappa$ and every integer $m$,
the following equality holds:
$$
\int\limits_{-\kappa}^{1-\kappa}e(\alpha m)d\alpha=
\left\{
\begin{array}{ll}
1, & \hbox{if } m=0,\\
0, & \hbox{if } m\ne 0.
\end{array}
\right.
$$
\end{lemma}

\begin{lemma}\label{osrzn} {\rm \cite{Rakhmonovsredzn2023}.} For $x\geq x_0>0$, $\sqrt{x}<y\leq
0.01x$, the estimate
 $$
 \int_{0}^{1}|T_n(\alpha;x,y)|^{2^k}d\alpha\ll y^{2^k-k+\varepsilon}, 1\le k\le n
 $$
holds, where
 \[
T_n(\alpha;x,y)=\sum_{x-y<m\le x}e(\alpha m^n).
\]
\end{lemma}
\begin{lemma}\label{lemmaDirichle} {\rm \cite{KarOATCh}.} \textbf{(Dirichlet's rational approximation lemma).}
For every real number \(\alpha\) and every parameter \(P\geq 1\),
there exist coprime integers \(a\) and \(q\), $1\leq q\leq P,$
such that
\[
\left|\alpha-\frac{a}{q}\right|\leq \frac{1}{qP}.
\]
\end{lemma}
In Lemmas 4 and 5, we use the following notation and assumptions:
\[
x\ge x_0>0,\  
0<y\le 0.01x,\  
\alpha=\frac{a}{q}+\lambda, \ 
(a,q)=1, 
q\le\tau, 
|\lambda|\le\frac{1}{q\tau}.
\] 
\begin{lemma}\label{exposumm1} {\rm \cite{RakhmonovAzamovMirzo2010}.}
Let $\tau\geq 24x^2y$ and
$|\lambda|\le \frac{1}{8q x^3}$. Then
$$
T(\alpha,x,y)=\frac{y}{q}S(a,q)\gamma(\lambda;x,y)+O(q^{\frac{1}{2}+\varepsilon
}), 
$$
$$ 
\gamma(\lambda;x,y)=\int\limits_{-\frac{1}{2}}^{\frac{1}{2}}e\Bigl(\lambda(x-\frac{y}{2}+yt)^4\Bigr)dt.
$$
\end{lemma}
 
\begin{lemma}\label{exposumm2} {\rm \cite{RakhmonovAzamovMirzo2010}.}
Let $\tau\geq 24x^2y$ and $\frac{1}{8q x^3}<|\lambda|\le
 \frac{1}{q\tau}$. Then
$$
T(\alpha,x,y)\ll q^{\frac{3}{4}}\ln q+q^{\frac{1}{4}}x^{\frac{1}{2}}.
$$
\end{lemma}
\begin{lemma}\label{pervproiz} {\rm \cite{Archipov}.}
Let \(f\) be a real-valued function on the interval \([a,b]\), with
\(f'(u)\) monotonic and $|f'(u)|\geq m>0$ for $a\leq u\leq b$.
Let \(g\) also be a monotonic function on \([a,b]\) such that
$|g(u)|\leq M.$
Then
\[
\int\limits_{a}^{b} g(u)e(f(u))\,du \ll \frac{M}{m}.
\]
\end{lemma}
\begin{lemma}\label{vtorproiz} {\rm \cite{Archipov}.}
Let \(f\) be a real-valued function twice differentiable on the
interval \([a,b]\). If, for all \(u\in(a,b)\),
\[
|f''(u)|\geq \lambda>0,
\]
then
\[
\int_a^b e(f(u))\,du \ll \lambda^{-\frac{1}{2}}.
\]
\end{lemma}
\begin{lemma}\label{orsumm} {\rm \cite{Archipov}.}
Let $n \geq 3$ be an integer, and let
\[
    f(x)=a_nx^n+\cdots+a_1x+a_0
\]
be a polynomial with integer coefficients. Let $q$ also be a natural number and
\[
    (a_n,\ldots,a_1,q)=1.
\]
Then
\[
    |\sum\limits_{x=1}^{q}e\Big(\frac{f(x)}{q}\Big)| \leq c(n) q^{1-\frac{1}{n}},
\]
where
\[
    c(n)=
    \begin{cases}
        e^{4n}, & n\geq 10,\\[2mm]
        e^{nA(n)}, & 3\leq n\leq 9,
    \end{cases}
\]
and
\[
\begin{aligned}
    A(3)&=6.1, & A(4)&=5.5, & A(5)&=5, & A(6)&=4.7,\\
    A(7)&=4.4, & A(8)&=4.2, & A(9)&=4.05.
\end{aligned}
\]
\end{lemma}

\begin{lemma}\label{puasson} {\rm \cite{Von}.}
Let \(f\) be a real-valued function on \([a,b]\) such that
\(f'(u)\) is monotonic. If, for every $u\in[a,b]$ and some integers \(H_1,H_2\), one has $H_1\leq f'(u)\leq H_2,$
then
\[
\sum_{a<n\leq b} e(f(n))
=
\sum_{H_1\leq h\leq H_2}
\int_a^b e(f(u)-hu)\,du
+
O(\log(H+2)),
\]
where \(H=\max(|H_1|,|H_2|)\).
\end{lemma}

\begin{lemma}\label{Veyl} {\rm \cite{Hua}.} Let $x\ge x_0>0$,
$\sqrt{x}<y\le 0.01x$, let $\alpha$ be a real number, and suppose that
$|\alpha-\frac{a}{q}|\le \frac{1}{q^2}$, $(a,q)=1$. Then
$$
|T(\alpha;x,y)|\le
6y^{1+\varepsilon}\left(\frac{1}{y}+\frac{1}{q}+\frac{q}{y^4}\right)^{\frac{1}{8}}.
$$
\end{lemma}
\section{The intermediate-range estimate for short quartic Weyl sums}
\begin{proof}[Proof of Theorem~\ref{thm:quartic-sum}]
We have
\[
T\left(\frac{a}{q}+\lambda;x,y\right)
=
\sum_{x-y<n\le x}
e\left(\frac{an^4}{q}+\lambda n^4\right).
\]
We use the expansion
\[
e\left(\frac{an^4}{q}\right)
=
\frac{1}{q}\sum_{b=1}^{q}
S_b(a,q)e\left(-\frac{bn}{q}\right),
\]
where
\[
S_b(a,q)=
\sum_{v=1}^{q}
e\left(\frac{av^4+bv}{q}\right).
\]
Then
\[
T\left(\frac{a}{q}+\lambda;x,y\right)
=
\frac{1}{q}
\sum_{b=1}^{q}
S_b(a,q)T_b,
\]
where
\[
T_b=
\sum_{x-y<n\le x}
e\left(\lambda n^4-\frac{b}{q}n\right).
\]

Put
\[
f_b(u)=\lambda u^4-\frac{b}{q}u.
\]
Then
\[
f_b'(u)=4\lambda u^3-\frac{b}{q}.
\]
Since \(y\le 0.01x\), on the interval \(x-y\le u\le x\)
we have \(u\asymp x\), and the function \(f_b'(u)\) is monotonic; therefore,
by Lemma \ref{puasson},
\begin{equation}
\label{osnlemma01}
T_b
=
\sum_{h=H_{1b}}^{H_{2b}}
\int_{x-y}^{x}
e\left(\lambda u^4-\frac{b}{q}u-hu\right)\,du
+
O\bigl(\ln(H_b+2)\bigr),
\end{equation}
where $H_{1b}\le f_b'(u)\le H_{2b}$ for all $u\in[x-y,x]$, and
\[
H_b=\max(|H_{1b}|,|H_{2b}|).
\]

Since, for all $u\in[x-y,x]$,
\[
|f_b'(u)|
\le
4|\lambda|x^3+\frac{b}{q},
\]
taking into account $b\le q$ and choosing $H_{1b}, H_{2b}$ to be integers, we obtain
\[
H_b\ll 1+|\lambda|x^3.
\]

In (\ref{osnlemma01}), after the substitution $m=b+qh$, we have
$$
\frac {b}{q}+h=\frac{m}{q},
$$
and the corresponding integral has the form
\[
I_m=
\int_{x-y}^{x}
e\left(\lambda u^4-\frac{m}{q}u\right)\,du.
\]
Since \(m\equiv b \pmod{q}\),
\[
S_b(a,q)=S_m(a,q),
\]
where
\[
S_m(a,q)=
\sum_{v=1}^{q}
e\left(\frac{av^4+mv}{q}\right).
\]
Consequently, applying Lemma \ref{orsumm} to $S_m(a,q)$, we obtain
\begin{equation}
\label{osnlemma02}
T\left(\frac{a}{q}+\lambda;x,y\right)
\ll
q^{-\frac{1}{4}}
\sum_{m\in\mathcal M}
|I_m|
+
q^{\frac{3}{4}}\ln(2+|\lambda|x^3),
\end{equation}
where
\[
\mathcal M
=
\left\{
m=b+qh:\ 1\le b\le q,\ H_{1b}\le h\le H_{2b}
\right\}.
\]

Note that the passage from the double sum over \(b,h\) to a sum over \(m\)
is valid, since, for \(1\le b\le q\), the representation of each such
\(m\) in the form $m=b+qh$
is unique.

We now estimate the sum of integrals
\begin{equation*}
\sum_{m\in\mathcal M}|I_m|,
\end{equation*}
where
\[
I_m=
\int\limits_{x-y}^{x}
e\left(\lambda u^4-\frac{m}{q} u\right)\,du.
\]
Let
\[
F_m(u)=\lambda u^4-\frac{m}{q} u.
\]
Then
\[
F_m'(u)=4\lambda u^3-\frac{m}{q},
\qquad
F_m''(u)=12\lambda u^2.
\]
Since \(y\le 0.01x\), on the interval \(x-y\le u\le x\)
we have \(u\asymp x\). Therefore,
\[
|F_m''(u)|\asymp |\lambda|x^2.
\]
The stationary points of the phase function \(F_m(u)\) are determined by the equation
\begin{equation*}
F_m'(u)=0
\end{equation*}
which is equivalent to
\[
m=4q\lambda u^3.
\]
Consequently, if a stationary point lies in the interval
\[
x-y\le u\le x,
\]
then the corresponding value \(m\) must belong to the interval
\[
\Delta=
\left[
\min\left(4q\lambda(x-y)^3,\;4q\lambda x^3\right),
\max\left(4q\lambda(x-y)^3,\;4q\lambda x^3\right)
\right].
\]
Conversely, if \(m\in\Delta\), then the equation
\[
m=4q\lambda u^3
\]
has a solution \(u\in[x-y,x]\), that is, the integral \(I_m\) has a stationary point.
Consequently, the integral \(I_m\) has a stationary point on the interval
\([x-y,x]\) if and only if \(m\in\Delta\).
The length of this interval satisfies the estimate
\[
|\Delta|
\ll
q|\lambda|\bigl(x^3-(x-y)^3\bigr).
\]
Since
\[
x^3-(x-y)^3\ll x^2y,
\]
we have
\[
|\Delta|\ll q|\lambda|x^2y.
\]

We distinguish the set of stationary and nearly stationary values
\[
\mathcal M_0=
\{m\in\mathcal M:\operatorname{dist}(m,\Delta)\le 1\},
\]
where \(\operatorname{dist}(m,\Delta)\) denotes the distance from the point \(m\)
to the interval \(\Delta\).\\
If \(\Delta=[\Delta_1,\Delta_2]\), then
\[
\operatorname{dist}(m,\Delta)=
\begin{cases}
0, & \Delta_1\le m\le \Delta_2,\\
\Delta_1-m, & m<\Delta_1,\\
m-\Delta_2, & m>\Delta_2.
\end{cases}
\]

Then
\[
|\mathcal M_0|\ll 1+|\Delta|
\ll 1+q|\lambda|x^2y.
\]
For \(m\in\mathcal M_0\), we apply Lemma \ref{vtorproiz}. Since
\[
|F_m''(u)|\asymp |\lambda|x^2,
\]
we obtain
\[
|I_m|\ll (|\lambda|x^2)^{-\frac{1}{2}}.
\]
Consequently,
\[
\sum_{m\in\mathcal M_0}|I_m|
\ll
(1+q|\lambda|x^2y)(|\lambda|x^2)^{-\frac{1}{2}}.
\]

Now consider \(m\in\mathcal M\setminus\mathcal M_0\).
Put
\[
r=\operatorname{dist}(m,\Delta).
\]

Then \(r>1\). For all \(u\in[x-y,x]\), the quantity
\(4q\lambda u^3\) belongs to the interval \(\Delta\); therefore,
\[
|4q\lambda u^3-m|\ge r.
\]
Consequently,
\[
|F_m'(u)|
=
\left|4\lambda u^3-\frac{m}{q}\right|
=
\frac{1}{q} |4q\lambda u^3-m|
\ge
\frac{r}{q}.
\]
Since \(F_m'(u)\) is monotonic on the interval \([x-y,x]\), Lemma \ref{pervproiz} gives
\[
|I_m|\ll \frac{q}{r}.
\]

The choice of the integers \(H_{1b},H_{2b}\), bounding the values of
\(f'_b(u)\) on the interval \([x-y,x]\), implies that, for all
\(m=b+qh\in\mathcal M\), the inequalities
\[
\Delta_1-q\le m\le \Delta_2+q.
\]
Indeed,
\[
\frac{m}{q}=\frac{b}{q}+h
\]
differs by at most a constant from the values of the function
\(4\lambda u^3\) on the interval \([x-y,x]\), and multiplication by \(q\)
gives a margin of \(O(q)\).

Consequently,
the remaining values \(m\in\mathcal M\setminus\mathcal M_0\) lie in
the two intervals
\[
\Delta_1-q\le m<\Delta_1-1
\]
or
\[
\Delta_2+1<m\le \Delta_2+q.
\]
Therefore,
\[
1<r\ll q.
\]
Summing over the distance \(r\), we obtain
\[
\sum_{m\in\mathcal M\setminus\mathcal M_0}|I_m|
\ll
\sum_{1< r\ll q}\frac{q}{r}
\ll
q\ln q.
\]

Thus,
\[
\sum_{m\in\mathcal M}|I_m|
\ll
q\ln q
+
(1+q|\lambda|x^2y)(|\lambda|x^2)^{-\frac{1}{2}}.
\]

Substituting this into (\ref{osnlemma02}), we obtain the following estimate for \(T\):
\[
T\left(\frac{a}{q}+\lambda;x,y\right)
\ll
q^{\frac{3}{4}}\ln q
+
q^{-\frac{1}{4}}(|\lambda|x^2)^{-\frac{1}{2}}
+
q^{\frac{3}{4}}y(|\lambda|x^2)^{\frac{1}{2}}
+
q^{\frac{3}{4}}\ln(2+|\lambda|x^3).
\]

From the lower bound \(
|\lambda|\gg \frac{1}{qx^2y}
\)
it follows that
\(
|\lambda|x^2\gg \frac{1}{qy}.
\)
Therefore,
\[
q^{-\frac{1}{4}}(|\lambda|x^2)^{-\frac{1}{2}}
\ll
q^{-\frac{1}{4}}(qy)^{\frac{1}{2}}
=
q^{\frac{1}{4}}y^{\frac{1}{2}}
\le
q^{\frac{1}{4}}\frac{x}{y^{\frac{1}{2}}}.
\]
On the other hand, from the upper bound
\(
|\lambda|\le \frac{1}{qy^3}
\)
we obtain
\(
|\lambda|x^2\le \frac{x^2}{qy^3}.
\)
Consequently,
\[
q^{\frac{3}{4}}y(|\lambda|x^2)^{\frac{1}{2}}
\le
q^{\frac{3}{4}}y
\left(\frac{x^2}{qy^3}\right)^{\frac{1}{2}}
=
q^{\frac{1}{4}}\frac{x}{y^{\frac{1}{2}}}.
\]
Thus,
\[
T\left(\frac{a}{q}+\lambda;x,y\right)
\ll
q^{\frac{3}{4}}\ln q
+
q^{\frac{3}{4}}\ln(2+|\lambda|x^3)
+
q^{\frac{1}{4}}\frac{x}{\sqrt{y}}.
\]

This completes the proof of Theorem~\ref{thm:quartic-sum}.
\end{proof}
\section{Application to Waring's problem}
\begin{proof}[{\bf Proof of Theorem \ref{varing4}}]
Assume that
$N^{\frac{6}{25}+\varepsilon}\leq H\leq N^{\frac{1}{4}-\varepsilon}$.
Let $Q=\dfrac{H}{L}$,
$\tau~=~48(N_1+H)^2H$, $\kappa \tau =1$, $E=[-\kappa ,1-\kappa ]$.

Applying Lemma \ref{int1}, we have
 $$
 J(N,H)= \sum_{|x_i-N_1|\le H,\atop
 i=1,\ldots, 17}\quad
 \int_{-\kappa}^{1-\kappa}e(\alpha(x_{1}^4+x_{2}^4+\ldots+x_{17}^4-N))d\alpha=
 $$
$$
=\int\limits_E\left(\sum_{|n-N_1|\le H}e(\alpha n^4)\right)^{17}
e(-\alpha N)d\alpha=
$$
$$
=\int\limits_E\left(T(\alpha ;N_1+H,2H)+\theta\right)^{17} e(-\alpha
N)d\alpha ,
$$

where
\[
\theta=\theta(\alpha)=
\begin{cases}
e\!\left(\alpha(N_1-H)^4\right), & N_1-H\in\mathbb Z,\\
0, & N_1-H\notin\mathbb Z.
\end{cases}
\]
Thus, \(|\theta|\leq 1\).

Using the relation
$$
\left(T(\alpha ;N_1+H,2H)+\theta\right)^{17}-T^{17}(\alpha
;N_1+H,2H)\ll\left|T(\alpha ;N_1+H,2H)\right|^{16}+1,
$$
and Lemma \ref{osrzn} with $k=4$, we find
\begin{align*}
&\int\limits_E\left|T(\alpha ;N_1+H,2H)\right|^{16}d\alpha\ll
H^{12+\varepsilon}=\frac{H^{16}}{N^{\frac{3}{4}}L^{16}}\cdot \frac{N^{\frac{3}{4}}L^{16}}{H^{4-\varepsilon}}\le \\
&\leq \frac{H^{16}}{N^{\frac{3}{4}}L^{16}}\cdot
N^{-\frac{21}{100}-\frac{94}{25}\varepsilon+\varepsilon^2}L^{16}\ll
\frac{H^{16}}{N^{\frac{3}{4}}L^{16}}.
\end{align*}

Therefore,
$$
J(N,H)=\int\limits_ET^{17}(\alpha; N_1+H,2H)e(-\alpha N)d\alpha
+O\left(\frac{H^{16}}{N^{\frac{3}{4}}L^{16}}\right).
$$

By Dirichlet's lemma (Lemma \ref{lemmaDirichle}) on approximation
of real numbers by rational numbers, every $\alpha $ in
the interval $E$ can be represented in the form
\begin{equation}\label{osntheor1}
 \alpha =\frac{a}{q}+\lambda ,\quad (a,q)=1,\quad 1\le
q\le \tau , \quad |\lambda |\le \frac{1}{q\tau }.
\end{equation}
It is easy to see that, in this representation, $0\le a\le q-1$, with
$a=0$ only when $q=1$. 

We define 
\[
E_1=
\bigcup_{1\le q\le Q}
\ \bigcup_{\substack{0\le a<q\\(a,q)=1}}
\left\{
\alpha\in E:
\left|\alpha-\frac{a}{q}\right|
\le\frac{1}{q\tau}
\right\},
\quad
E_2=E\setminus E_1.
\]
The set $E_1$ consists of
pairwise disjoint intervals. Indeed, $E_1$ consists of
intervals $E(a,q)$ of the form
\begin{align*}
\frac{a}{q}&-\frac{1}{q\tau}\le \alpha\le
\frac{a}{q}+\frac{1}{q\tau},\ 0\le a<q,\\
& (a,q)=1, q=1,2,\ldots, [Q].
\end{align*}
If $E(a,q)$ and $E(a_1,q_1)$ are two different intervals of $E_1$, that is,
$(a-a_1)^2+(q-q_1)^2\ne 0$, then the distance between the centers of these
intervals is
$$
\left|\frac{a}{q}-\frac{a_1}{q_1}\right|\ge\frac{1}{qq_1},
$$
whereas the sum of their half-lengths is
$$
\frac{1}{q\tau}+\frac{1}{q_1\tau}<\frac{1}{qq_1}.
$$
Consequently, $E(a,q)$ and $E(a_1,q_1)$ do not intersect.

We divide the set $E_1$ into the sets
$E_{11}$ and $E_{12}$:
\begin{align*}
 & E_{11}=\left\{ \alpha :\ \alpha \in E_1,\ \left|\alpha
-\frac{a}{q}\right|\le \delta \right\}, \quad \delta=
\frac{1}{8q(N_1+H)^3};  \\
 &   E_{12}=\left\{ \alpha :\ \alpha \in E_1,\ \delta<\left|\alpha -\frac{a}{q}\right|\le \frac{1}{q\tau}\right\} .
\end{align*}

Let $J_{11}$,
$J_{12}$, and $J_2$ denote, respectively, the integrals over the sets
$E_{11}$, $E_{12}$, and $E_2$. We then have
\begin{equation}\label{osntheor2}
J(N,H)=J_{11}+J_{12}+J_2+O\left(\frac{H^{16}}{N^{\frac{3}{4}}L^{16}}\right).
\end{equation}
In the decomposition (\ref{osntheor2}), the integral \(J_{11}\) gives the main term
of the asymptotic formula for \(J(N,H)\), whereas \(J_{12}\) and \(J_2\)
must be estimated as parts of the remainder term.

The argument of \cite{Azamov2011}, based on Lemmas \ref{exposumm1} and \ref{exposumm2}, yields
the following estimates for the integrals \(J_{11}\) and \(J_{12}\):
  \begin{equation}
\label{osntheor10} J_{11}=\frac{B\mathfrak{S}(N)H^{16}}{N^{\frac{3}{4}}}\left(1+O\left(L^{-16}\right)\right).
\end{equation}
 \begin{equation}
\label{I12} J_{12}\ll\frac{H^{16}}{N^{\frac{3}{4}}L^{16}}.
\end{equation}
Repeating the proof of the corresponding estimates from~\cite{Azamov2011} and
noting that, for
\[
 H\ge N^{\frac{6}{25}+\varepsilon}
\]
the conditions for applying Lemmas \ref{exposumm1} and \ref{exposumm2}, as well as the relations
between the parameters used there, remain valid, we obtain that the estimates
(\ref{osntheor10}) and (\ref{I12}) remain valid in the case under
consideration.

However, direct application of the estimates from~\cite{Azamov2011} to the integral \(J_2\) in
the case under consideration no longer gives the required estimate of the form
\[
 J_2\ll \frac{H^{16}}{N^{\frac{3}{4}}L^{16}}.
\]
Therefore, to complete the proof, it is necessary to obtain a sharper
estimate for the integral \(J_2\). The remainder of the
proof is devoted to this estimate.

\textbf{Estimate for the integral $J_2$.} We have
\begin{equation*}
\label{osntheor13}
|J_{2}|\le 
\int\limits_{E_2}\left|T(\alpha; N_1+H,2H)\right|^{17}d\alpha .
\end{equation*}

We estimate $T(\alpha ;N_1+H,2H)$
for $\alpha $ in the set $E_2$.

For each \(\alpha\in E_2\), apply Lemma \ref{lemmaDirichle} with the parameter
\[
P=H^3.
\]
Then there exist coprime integers \(a,r\) such that
\[
\alpha=\frac{a}{r}+\beta,
\qquad
1\le r\le H^3,
\qquad
|\beta|\le \frac{1}{rH^3}.
\]

We divide \(E_2\) into two parts:
\[
E_2=E_{21}\cup E_{22},
\]
where
\[
E_{21}=\{\alpha\in E_2:\ r>Q\},
\]
and
\[
E_{22}=\{\alpha\in E_2:\ r\le Q\}.
\]
Accordingly,
\begin{equation}
\label{osntheor14}
J_2=J_{21}+J_{22},
\end{equation}
where
\[
|J_{2j}|\le\int\limits_{E_{2j}}|T(\alpha; N_1+H,2H)|^{17}\,d\alpha,
\qquad j=1,2.
\]
\textbf{Estimate for the integral $J_{21}$.}

Let \(\alpha\in E_{21}\). Then
\[
Q<r\le H^3.
\]
By Weyl's inequality (Lemma \ref{Veyl}) for a short quartic sum, we have
\[
T(\alpha)\ll
H^{1+\varepsilon}
\left(
\frac{1}{r}+\frac{1}{H}+\frac{r}{H^4}
\right)^{1/8}.
\]
Since
\[
r>Q=\frac{H}{L},
\]
we have
\[
\frac{1}{r}\le \frac{L}{H}.
\]
Moreover, \(r\le H^3\) implies
\[
\frac {r}{H^4}\le \frac{1}{H}.
\]
Consequently,
\[
\frac{1}{r}+\frac{1}{H}+\frac{r}{H^4}
\ll
\frac{L}{H}.
\]
Hence
\[
T(\alpha)
\ll
H^{1+\varepsilon}
\left(\frac{L}{H}\right)^{\frac{1}{8}}
\ll
H^{\frac{7}{8}+\varepsilon}.
\]
Consequently,
\[
\max_{\alpha\in E_{21}}|T(\alpha ;N_1+H,2H)|
\ll
H^{\frac{7}{8}+\varepsilon},
\]
\[
|J_{21}|
\le
\left(\max_{\alpha\in E_{21}}|T(\alpha ;N_1+H,2H)|\right)
\int\limits_E |T(\alpha ;N_1+H,2H)|^{16}\,d\alpha.
\]
Applying Lemma \ref{osrzn}, we obtain
\[
\int\limits_E |T(\alpha ;N_1+H,2H)|^{16}\,d\alpha
\ll H^{12+\varepsilon}.
\]
Consequently,
\begin{equation}
\label{osntheor15}
J_{21}
\ll
H^{\frac{7}{8}+\varepsilon}H^{12+\varepsilon}
\ll
H^{\frac{103}{8}+\varepsilon}.
\end{equation}

\textbf{Estimate for the integral $J_{22}$.}
We now estimate the integral
\[
|J_{22}|\le \int\limits_{E_{22}}|T(\alpha; N_1+H,2H)|^{17}\,d\alpha .
\]
Let
\[
N_1=\left(\frac {N}{17}\right)^{\frac{1}{4}},\qquad
T(\alpha;N_1+H,2H)
=\sum_{N_1-H<n\le N_1+H}e(\alpha n^4).
\]

In the notation of Theorem \ref{thm:quartic-sum}, put
\[
x=N_1+H,\qquad y=2H.
\]
Then
\[
x^2y=2(N_1+H)^2H.
\]

Let \(\alpha\in E_{22}\). Then, by the definition of the set \(E_{22}\),
we have
\[
\alpha=\frac{a}{r}+\beta,\qquad (a,r)=1,
\]
where
\[
r\le Q,\qquad |\beta|\le \frac{1}{rH^3}.
\]
Since \(\alpha\notin E_1\),
\[
|\beta|>\frac{1}{r\tau}.
\]
With our choice,
\[
\tau=48(N_1+H)^2H=24x^2y.
\]
Consequently,
\[
|\beta|>\frac{1}{24rx^2y},
\]
that is,
\[
|\beta|\gg \frac{1}{rx^2y}.
\]
On the other hand, since \(y=2H\), the condition
\[
|\beta|\le \frac{1}{rH^3}
\]
implies
\[
|\beta|\le \frac{8}{ry^3}\ll \frac{1}{ry^3}.
\]
Thus the hypotheses of Theorem~\ref{thm:quartic-sum} are satisfied,
and therefore
\[
T\left(\frac{a}{r}+\beta;N_1+H,2H\right)
\ll
r^{\frac{3}{4}}\ln r
+
r^{\frac{3}{4}}\ln(2+|\beta|N_1^3)
+
r^{\frac{1}{4}}\frac{N_1}{H^{1/2}}\ll Q^{\frac{3}{4}}\ln N+ Q^{\frac{1}{4}}\frac{N^{\frac{1}{4}}}{H^{\frac{1}{2}}}.
\]
Since $Q=\frac{H}{L}$,
\[
Q^{\frac{3}{4}}L
\ll
H^{\frac{3}{4}}L^{\frac{1}{4}}
\ll
H^{\frac{7}{8}}.
\]
The condition \(H\ge N^{\frac{6}{25}+\varepsilon}\) implies
\[
Q^{\frac{1}{4}}\frac{N^{\frac{1}{4}}}{H^{\frac{1}{2}}}
\ll
H^{-\frac{1}{4}}N^{\frac{1}{4}}
=
H^{\frac{7}{8}}H^{-\frac{9}{8}}N^{\frac{1}{4}}
\ll
H^{\frac{7}{8}}N^{-\frac{1}{50}-\frac{9}{8}\varepsilon}
\ll
H^{\frac{7}{8}}.
\]

Thus,
\[
\max_{\alpha\in E_{22}}|T(\alpha; N_1+H,2H)|
\ll
H^{\frac{7}{8}}.
\]

Next, using Lemma \ref{osrzn}, we have
\begin{equation}
\label{osntheor16}
|J_{22}|
\le
\int\limits_{E_{22}}|T(\alpha)|^{17}\,d\alpha
\le
\max_{\alpha\in E_{22}}|T(\alpha)|
\int_E |T(\alpha)|^{16}\,d\alpha\ll H^{\frac{7}{8}}H^{12+\varepsilon}\ll H^{\frac{103}{8}+\varepsilon}.
\end{equation}

Substituting the estimates (\ref{osntheor15}) and (\ref{osntheor16}) into (\ref{osntheor14}), we find
\begin{equation}
\label{osntheor17}
J_{2}\ll H^{\frac{103}{8}+\varepsilon}=\frac{H^{16}}{N^{\frac{3}{4}}L^{16}}\cdot \frac{N^{\frac{3}{4}}L^{16}}{H^{\frac{25}{8}-\varepsilon}} \ll \frac{H^{16}}{N^{\frac{3}{4}}L^{16}}\cdot N^{-\frac{577}{200}\varepsilon+\varepsilon^2}L^{16}\ll \frac{H^{16}}{N^{\frac{3}{4}}L^{16}}
\end{equation}

Substituting the estimates obtained for $J_{11}$, $J_{12}$,
and $J_2$ from (\ref{osntheor10}), (\ref{I12}), and (\ref{osntheor17}), respectively, into (\ref{osntheor2}),
we obtain the assertion of the theorem.
\end{proof}

\end{document}